\documentclass[a4paper,10pt]{amsart}
\usepackage{}
\usepackage{txfonts,dsfont,amsfonts}
\usepackage{mathrsfs}
\usepackage{amsmath,amssymb,amscd,bbm,amsthm}

\numberwithin{equation}{section}

\theoremstyle{plain}
\newtheorem{theorem}{Theorem}[section]
\newtheorem{corollary}[theorem]{Corollary}
\newtheorem{lemma}[theorem]{Lemma}

\newtheorem{proposition}[theorem]{Proposition}
\theoremstyle{definition}
\newtheorem{definition}[theorem]{Definition}
\newtheorem{example}[theorem]{Example}

\theoremstyle{remark}
\newtheorem{remark}[theorem]{Remark}

\DeclareMathOperator{\gldim}{\rm gldim}

\DeclareMathOperator{\lw}{\rm
LW}\DeclareMathOperator{\lc}{\rm LC}
\DeclareMathOperator{\ext}{\underline{Ext}}
\DeclareMathOperator{\tor}{\underline{Tor}}
\DeclareMathOperator{\gkdim}{\rm GKdim }
\DeclareMathOperator{\id}{\rm id }
\DeclareMathOperator{\sh}{\rm Sh }

\begin{document}

\title{Lyndon words for Artin-Schelter regular algebras}

\author{G.-S. Zhou\; and \; D.-M. Lu}

\address{Department of Mathematics, Zhejiang University,
Hangzhou 310027, China}

\email{10906045@zju.edu.cn; \quad dmlu@zju.edu.cn }

\begin{abstract}
We show certain invariants of graded algebras of which all obstructions are Lyndon words and provide some methods to construct Artin-Schelter regular algebras from a closed set of Lyndon words.
\end{abstract}

\thanks {This research is supported by the NSFC (Grant No. 11271319).}

% 16E10 (1991-now) Homological dimension
% 16E65 (2000-now) Homological conditions on rings
%       (generalizations of regular, Gorenstein, Cohen-Macaulay rings, etc.)
% 16S15 (2000-now) Finite generation, finite presentability, normal forms (diamond lemma, term-rewriting)
% 16W50 (1991-now) Graded rings and modules
% 14A22 (1991-now) Noncommutative algebraic geometry
% 16E45  (2000-now) Differential graded algebras and applications
% 68R15 (2000-now) Combinatorics on words


\maketitle

%\bigskip

%\tableofcontents

%\bigskip

%\newpage

\section*{Introduction}

\newtheorem{maintheorem}{\bf{Theorem}}
\renewcommand{\themaintheorem}{\Alph{maintheorem}}
\newtheorem{maincorollary}[maintheorem]{\bf{Corollary}}
\renewcommand{\themaincorollary}{}
\newtheorem{mainquestion}[maintheorem]{\bf{Question}}
\renewcommand{\themainquestion}{}

A Lyndon word is a non-empty word greater in lexicographical order than all of its rotations. Lyndon words share remarkable combinatoric and algebraic properties. For instance, they are wildly used in the context of Lie algebras and their universal enveloping algebras. Recently, monomial algebras defined by Lyndon words are studied in \cite{GF}. It will be interesting to explore Lyndon words in a more general context of associative algebras. This work is a practice of the idea for Artin-Schelter regular algebras.

Artin-Schelter regular algebras are a class of graded algebras which may be thought of as homogeneous coordinate rings of noncommutative spaces. They were introduced by Artin and Schelter \cite{ASc} in late 1980's, and have been extensively studied since then. The class of such algebras of which all the obstructions are Lyndon words is of particular interest. This class includes all low dimensional Artin-Schelter regular bigraded algebras in two generators \cite[Theorem 8.1]{ZL} and the universal enveloping algebras of positively graded finite dimensional Lie algebras. Invariants of algebras in this class are relatively easy to deal with, one of our goals is to find accessible approaches for constructing of new such algebras.

It is well-known that the standard bracketing of Lyndon words on an alphabet $X$ form a basis of the free Lie algebra ${\rm Lie}(X)$, and their monotonic products form a basis of the free associative algebra $k\langle X\rangle$. Given a Gr\"{o}bner set $G$ of Lie polynomials, then the algebra $k\langle X\rangle/(G)$, which is the universal enveloping algebra of the Lie algebra $\mathfrak{g}=\text{Lie}(X)/(G)_L$, is determined by irreducible Lyndon words modulo $G$ (see \cite{LR,BC}). The main tool of our approaches is a deformation of the standard bracketing of Lyndon words through a matrix $q$. We associate a closed set $U$ of Lyndon words with a graded algebra $A(U,q)$ (Definition \ref{Definition-A(U,q)}) and give a practical criterion for determining the regularity of $A(U,q)$. As an example, we are able to obtain the algebras $D(v,p)$ in \cite[Theorem A]{LPWZ} as a deformation of the universal enveloping algebra $D(-2,-1)$ in this sense, which is a motivation of this paper.

A natural question for Artin-Schelter regular algebras is the determination of their Gorenstein parameters. In \cite[Proposition 2.4]{FV}, the authors give an optimal estimation for the universal enveloping algebras of finite dimensional graded Lie algebras which are generated in degree one. In this paper, we obtain a counterpart of such estimation for a larger class of  Artin-Schelter regular algebras which are generated in degree one and of which all the obstructions are Lyndon words (Corollary \ref{Corollary-bound-Gorenstein-parameter}).

The organization of this paper is as follows. In Section 1, we introduce notations, recall the definitions of Lyndon words and Artin-Schelter regular algebras, and review some basic facts of Lyndon words. In Section 2, we give a picture of the graph of chains on antichains of Lyndon words, calculate algebraic and homological invariants of graded algebras of which the obstructions are Lyndon words. We devote Section 3 to the construction of Artin-Schelter regular algebras in terms of the generalized bracketing of Lyndon words. The results, in particular, give a connection between Artin-Schelter regular algebras and Hopf algebras.

We work over a fixed field $k$ of characteristic $0$. All vector spaces, algebras and unadorned tensor $\otimes$ are over $k$. The notation $\mathbb{N}$ denotes the set of non-negative integers.

\section{Preliminaries}

\subsection{Noncommutative Gr\"{o}bner Bases}
We begin by introducing some notations and terminologies that will be used in the sequel. For noncommutative Gr\"{o}bner bases theory, we refer the details to \cite{Li,Mora}.

Throughout $X=\{x_1,\cdots,x_n\}$ stands for a finite enumerated alphabet of letters. Denote by $X^*$ the set of all words on $X$ including the empty word 1. The \emph{length} of a word $u$ is denoted by $l(u)$, and the \emph{constitute} of $u$ is the tuple $ (r_1,\cdots,r_n)\in\mathbb{N}^n$, where $r_i$ is the number of occurrences of $x_i$ in $u$ for $1\leq i \leq n$, in particular, the constitute of $1$ is $(0,\cdots,0)$. With each $x_i$ we associate a positive integer $d_i$. Define the degree of $u$ by $\deg(u)=r_1d_1 +\cdots + r_nd_n$. The free algebra on $X$ is denoted by $k\langle X\rangle$, and it is considered as a graded algebra $k\langle X\rangle=\bigoplus_{m\geq0}k\langle X\rangle_m$, where $k\langle X\rangle_m$ is the linear span of all words of degree $m$.

We say that a word $u$ is a \textit{factor} of another word $v$ if $w_1uw_2=v$ for some words $w_1,w_2\in X^*$. If $w_1=1$ then $u$ is called a \textit{prefix} of $v$, and if $w_2=1$ then $u$ is called a \textit{suffix} of $v$. A factor $u$ of $v$ is called \textit{proper} if $u\neq v$. We fix an ordering $x_1<x_{2}<\cdots < x_n$ on the letters. The \emph{lexicographical order} (\emph{lex order}, for short) $<_{\rm lex}$ on $X^*$ is then given as follows: for $u, v\in X^*$, $u <_{\rm lex} v$ iff
\begin{itemize}
\item either there are factorizations $u=rx_is,v=rx_jt$ with $x_i<x_j$, or
\item $v$ is a proper prefix of $u$.
\end{itemize}
Clearly $<_{\rm lex}$ is a total order, which is preserved by left multiplication, but not always by right multiplication. For example, $x_2^2 <_{\rm lex}  x_2$, but $x_2^2x_1 >_{\rm lex} x_2x_1$. However, if $u<_{\rm lex}v$ and $v$ is not a prefix of $u$, then the order is preserved by right multiplication, even right multiplication by different words. In particular, this holds when $\deg(u)=\deg(v)$.

In the sequel we use the \emph{deglex} order $<_{\rm deglex}$ on $X^*$ given by, for $u, v\in X^*$,
\begin{eqnarray*}
\label{definition-deglex}
u <_{\rm deglex} v\quad \text{ iff }\quad\left\{
\begin{array}{llll}
\deg(u) < \deg(v),\quad \text{or}&&\\
\deg(u) = \deg(v)$ and $u<_{\rm lex}v.
\end{array}\right.
\end{eqnarray*}
It is an \emph{admissible order} on $X^*$, which means a well order on $X^*$ preserved by left and right multiplications. For any nonzero polynomial $f\in k\langle X\rangle$, the \textit{leading word} $\lw(f)$ of $f$ is the largest word occurs in $f$, and the \textit{leading coefficient} $\lc(f)$ of $f$ is the coefficient of $\lw(f)$ in $f$. A nonzero polynomial with leading coefficient $1$ is said to be \textit{monic}.

Let $G\subseteq k\langle X\rangle$ be a set of polynomials. The set of all leading words of nonzero polynomials in $G$ is denoted by $\lw(G)$. A word $u$ is \emph{reducible modulo $G$} (resp. \emph{irreducible modulo $G$}) provided that $u$ has a factor in $\lw(G)$ (resp. has no factor in $\lw(G)$). The set of all irreducible words modulo $G$ is denoted by ${\rm Irr}(G)$. If $0\not\in G$, every $f\in G$ is monic and all words occurred in $f$ are irreducible modulo $G\backslash\{f\}$, then we say $G$ is a \emph{reduced set}.

A $4$-tuple $(l_1,r_1,l_2,r_2)$ of words is called an \textit{ambiguity} of
a pair of words $(u_1,u_2)$ in case $l_1u_1r_1=l_2u_2r_2$ and
one of the following conditions holds: (1) $l_1=r_1=1$;
(2) $l_1=r_2=1$, $r_1$ is a nontrivial proper suffix of $u_2$
and $l_2$ is a nontrivial proper prefix of $u_1$.
We define a \textit{composition} of nonzero polynomials $f_1,f_2$ to be a polynomial of the form
\[S(f_1,f_2)[l_1,r_1,l_2,r_2]=\frac{l_1f_1r_1}{\lc(f_1)}-\frac{l_2f_2r_2}{\lc(f_2)},\] where $(l_1,r_1,l_2,r_2)$ is an ambiguity of $\big(\lw(f_1),\lw(f_2)\big)$.

\begin{definition}
Let $G$ be a subset of $k\langle X\rangle$. A polynomial $f$ is \emph{trivial modulo $G$} if it has a presentation
$$
f=\sum_{i\in I}a_iu_ig_iv_i,\quad a_i\in k,\; u_i, v_i\in X^*,\; g_i\in G\backslash\{0\}
$$
with $u_i\lw(g_i)v_i\leq_{\rm deglex}\lw(f)$. The set $G$ is called a \emph{Gr\"{o}bner set} in $k\langle X\rangle$ if any composition of nonzero polynomials in $G$ is trivial modulo $G$.
\end{definition}

Let $\mathfrak{a}$ be an ideal of $k\langle X\rangle$. A word $u$ is called an \emph{obstruction} of the algebra $A=k\langle X\rangle/\mathfrak{a}$ if $\lw(\mathfrak{a})$ contains $u$ but no proper factor of $u$. If $G$ is a Gr\"{o}bner set that generates $\mathfrak{a}$, then the set of obstructions of $A$ coincides with the set $\{u\in\lw(G)\,|\, u \text{ has no proper factor in }\lw(G)\}$. Note that there exists a unique reduced Gr\"{o}bner set $G_0$ that generates $\mathfrak{a}$. If $\mathfrak{a}$ is homogeneous, then $G_0$ consists of homogeneous polynomials and it contains a minimal generating set of $\mathfrak{a}$ (\cite[Lemma 2.1]{ZL}). The following version of Bergman's Diamond Lemma assures the usefulness of Gr\"{o}bner bases theory.

\begin{proposition}\cite[Theorem 1.2]{Berg}\label{Proposition-Diamond-Lemma}
Let $\mathfrak{a}$ be an ideal of $k\langle X\rangle$ and $G\subseteq \mathfrak{a}$. Then the following statements are equivalent:
\begin{enumerate}
\item $G$ is a Gr\"{o}bner set that generates $\mathfrak{a}$.
\item The leading word of every nonzero polynomial in $\mathfrak{a}$ is reducible modulo $G$.
\item Every polynomial in $\mathfrak{a}$ is trivial modulo $G$.
\item ${\rm Irr}(G)$ forms a basis of the quotient algebra $k\langle X\rangle/\mathfrak{a}$.
\end{enumerate}
\end{proposition}

\subsection{Lyndon Words}
In this subsection we give a short summary of facts about Lyndon words which we shall use. All of the facts in this subsection are well-known. Our main references are \cite{BC,GF,Kh,Lo}.

\begin{definition}
A word $u\in X^*$ is called a \textit{Lyndon word} if $u\neq1$ and $u>_{\rm lex} wv$ for every factorization $u=vw$ with $v,w\neq1$. The set of Lyndon words on $X$ is denoted $\mathbb{L}=\mathbb{L}(X)$.
\end{definition}

Immediately, by definition, there is no non-empty word that is both a proper  prefix and a proper suffix of a Lyndon word. In particular, every Lyndon word is not a power of another word. There are alternative characterizations for checking whether a word is Lyndon.

\begin{proposition}\label{Proposition-character-Lyndon}
The following statements are equivalent for any non-empty word $u$:
\begin{enumerate}
\item $u$ is a Lyndon word.
\item $u >_{\rm lex} w$ for every factorization $u=vw$ with $v,w \neq 1$.
\item $v >_{\rm lex} w$ for every factorization $u = vw$ with $v,w\neq1$.
\end{enumerate}
\end{proposition}

\begin{proof}
By \cite[Lemma 2]{Kh} one has $(1) \Leftrightarrow (2)$ and it is easy to see $(2) \Rightarrow (3)$. Assume that $(3)$ holds and let $u = vw$ with $v,w\neq1$ be an arbitrary factorization. Note that $v>_{\rm lex}w$.  If $v$ is not a prefix of $w$, then $u = vw >_{\rm lex} w$. Otherwise $v$ is a prefix of $w$, and then $w=v^kw'$ for some $k\geq1$ and $w'\in X^*$ such that $v$ is not a prefix of $w'$. Therefore $u=v^{k+1}w'$ and $v\geq_{\rm lex} v^{k+1} >_{\rm lex} w'$. Observe that $v^{k+1} >_{\rm lex} v^kw' = w$ and $v^{k+1}$ is not a prefix of $w$, one gets $u>_{\rm lex}w$. Now the implication $(3) \Rightarrow (2)$ is completed.
\end{proof}

\begin{example}
The Lyndon words of length five or less in two letters $x_1,x_2$ are as follows:
\begin{eqnarray*}
x_1,\,x_2,\, x_2x_1,\, x_2x_1^2,\, x_2^2x_1,\, x_2x_1^3,\, x_2^2x_1^2, \,x_2^3x_1,\, x_2x_1^4,\, x_2x_1x_2x_1^2,\, x_2^2x_1^3,\, x_2^2x_1x_2x_1, \,x_2^3x_1^2,\,x_2^4x_1.
\end{eqnarray*}

\end{example}

Lyndon words have many good combinatorial features. We list below the corresponding version of \cite[Lemma 4.3]{GF}, \cite[Proposition 5.1.3, Proposition 5.1.4, Theorem 5.1.5]{Lo} and \cite[Lemma 4.5]{GF} respectively in our context.

\begin{proposition}
\label{proposition-facts-Lyndon}
\begin{enumerate}
\item Let $u=w_1w_2,\, v=w_2w_3$ be Lyndon words. If $u >_{\rm lex} v$ (which holds in priori when $w_2\neq1$), then $w=w_1w_2w_3$ is a Lyndon word.
\item If $u=u'u''$ with $u$ a Lyndon word and $u''$ its longest proper suffix that is Lyndon, then $u'$ is a Lyndon word. We call the pair $\sh(u)=(u',u'')$ the {\rm Shirshov factorization} of $u$.
\item Suppose that $u >_{\rm lex} v$ are Lyndon words. Then the Shirshov factorization of $uv$ is $(u,v)$ iff either $u$ is a letter or $u''\leq_{\rm lex} v$ when $\sh(u)=(u',u'')$.
\item Every word $u\neq1$ can be written uniquely as a product $u=w_1\cdots w_r$ of Lyndon words with $w_1\leq_{\rm lex} \cdots \leq_{\rm lex}w_r$. We call such a decomposition the {\rm Lyndon decomposition} of $u$.
\item If $v$ is a Lyndon word and it is a factor of $u$ with Lyndon decomposition $u=w_1\cdots w_r$, then $v$ is a factor of one of the words $w_1,\cdots,w_r$.
\end{enumerate}
\end{proposition}

\begin{remark}
Parts $(1),(2)$ of Proposition \ref{proposition-facts-Lyndon} allow us to get every Lyndon word by starting with $X$ and concatenating inductively each pair of Lyndon words $v,w$ with $v>_{\rm lex}w$. Note that in the Lyndon decomposition $u=w_1\cdots w_r$, $w_r$ is the longest Lyndon suffix of $u$.
\end{remark}

\begin{lemma}\label{Lemma-Shirshov-factor-dichotomy}
Assume that $u$ is a Lyndon word with $\sh(u)=(u',u'')$. If $v$ is a Lyndon word that is a factor of $u$, then either $v$ is a factor of $u'$, or $v$ is a factor of $u''$, or $v$ is a prefix of $u$ such that $l(v)> l(u')$.
\end{lemma}

\begin{proof}
It is suffice to exclude the possibility that there are factorizations $u'=pl,\, v=lr,\, u''=rq$ with $p,l,r\neq1$. Otherwise one would have a proper Lyndon suffix $lu''$ of $u$ by the part (1) of Proposition \ref{proposition-facts-Lyndon}, which is impossible by the part $(2)$ of Proposition \ref{proposition-facts-Lyndon}.
\end{proof}

\subsection{Artin-Schelter regular algebras}
An \emph{Artin-Schelter regular algebra} (AS-regular, for short) is a positively $\mathbb{Z}$-graded algebra $A=\bigoplus_{m\geq 0}A_m$ which is connected ($A_0=k$) and satisfies the following three conditions:
\begin{enumerate}
\item[(AS1)]
$A$ has finite global dimension $d$;
\item[(AS2)]
$A$ has finite Gelfand-Kirillov dimension ($\gkdim$, for short);
\item[(AS3)]$A$ is Gorenstein; that is, for some $l\in \mathbb{Z}$,
\begin{eqnarray*}
\ext^i_A(k_A,A)=\left\{
\begin{array}{ll}
0,&i\neq d,\\
k(l),& i=d,
\end{array}\right.
\end{eqnarray*}
where $k_A$ is the trivial right $A$-module $A/A_{>0}$, and the notation $(l)$ is the degree $l$-shifting on graded modules. The index $l$ will be called the \textit{Gorenstein parameter} of $A$.
\end{enumerate}

All known examples of AS-regular algebras are strongly Noetherian, Auslander-regular and Cohen-Macaulay. We refer to \cite[Section 5]{ZZ2} for a review of the definitions.

\section{Invariants of graded algebras}

In this section we focus on the calculation of invariants of graded algebras of which all obstructions are Lyndon words. We provide the details in the way for a consistency. The readers will note some of familiar results, for example in \cite{GF,GI}, with different forms and different certificates.

\begin{definition}
We call a set $U$ of Lyndon words \emph{closed} if $U \supseteq X$ and $U$ contains each Lyndon word that is a factor of some word in $U$. Also we call a set $V$ of words an \emph{antichain} if any word in $V$ has no proper factor in $V$.
\end{definition}

Let $\mathcal{Y}$ be the set of all closed subsets of $\mathbb{L}$ and let $\mathcal{Z}$ be the set of all subsets of $\mathbb{L}$ that contain no letters. Define three set maps as follows:
\begin{eqnarray*}
\Phi: \mathcal{Y} \to \mathcal{Z}, & \quad &\Phi(U) = \{\ v\in \mathbb{L}\backslash U\;|\; v \text{ has no proper factor in }\mathbb{L}\backslash U\ \}. \\
\bar{\Phi}: \mathcal{Y} \to \mathcal{Z}, & \quad &\bar{\Phi}(U) = \{\ v\in \mathbb{L}\backslash U\;|\; v',v''\in U \text{ whenever }\sh(v)=(v',v'')\ \}. \\
\Psi: \mathcal{Z} \to \mathcal{Y}, &\quad &\Psi(U) = \{\ v\in \mathbb{L}\;\;|\;v \text{ has no factor in }U\ \}.
\end{eqnarray*}
Clearly $\Phi(U)$ is always an antichain for any $U\in \mathcal{Y}$, and $\Phi(\Psi(V))=V$ for any antichain $V\in \mathcal{Z}$. Also one has $\Phi(U)\subseteq \bar{\Phi}(U)$ for any $U\in \mathcal{Y}$ and $\Psi \circ \Phi = \Psi \circ \bar{\Phi} = \id_{\mathcal{Y}}$.

\begin{example}\label{example-closed-set}
We list some typical examples of closed subsets of $\mathbb{L}(\{x_1,x_2\})$, which occur in the presentation of low dimensional $\mathbb{Z}^2$-graded Artin-Schelter regular algebras \cite[Theorem 8.1]{ZL}.

{\small\begin{center}
\begin{tabular}{|c|l|l|}
\hline \textbf{Case} & \textbf{Closed Set} & $\mathbf\Phi$ \textbf{and} $\mathbf{\bar{\Phi}}$ \\
\hline (1)& $U_2=\{x_1, x_2\}$ & $\Phi(U_2)=\bar{\Phi}(U_2)=\{x_2x_1\}$ \\ \hline
(2)& $U_3=\{x_1, x_2x_1, x_2\}$ & $\Phi(U_3)=\bar{\Phi}(U_3)=\{x_2x_1^2, x_2^2x_1\}$ \\ \hline
(3)& $U_4=\{x_1, x_2x_1, x_2^2x_1, x_2\}$ & $\Phi(U_4)=\bar{\Phi}(U_4)= \{x_2x_1^2, x_2^2x_1x_2x_1, x_2^3x_1\}$ \\ \hline
(4)& $U_5=\{x_1, x_2x_1, x_2^2x_1x_2x_1, x_2^2x_1, x_2 \}$ & \parbox[t][9mm]{57mm}{$\Phi(U_5)=\{x_2x_1^2, x_2^2x_1x_2x_1x_2x_1, x_2^2x_1x_2^2x_1x_2x_1, x_2^3x_1\}\hskip2cm $ $\bar{\Phi}(U_5)= \{x_2x_1^2,\ x_2^2\!x_1\!x_2\!x_1\!x_2\!x_1,\ x_2^2\!x_1\!x_2^2\!x_1\!x_2\!x_1,\ x_2^3\!x_1\!x_2\!x_1,\ x_2^3x_1\}$} \\ \hline
(5)& $U_5'=\{x_1, x_2x_1, x_2^2x_1, x_2^3x_1, x_2\}$ & $\Phi(U_5')=\bar{\Phi}(U_5')= \{x_2x_1^2, x_2^2x_1x_2x_1, x_2^3x_1x_2^2x_1, x_2^4x_1\}$ \\ \hline
(6)& $U_5''\!= \{x_1, x_2x_1^2, x_2x_1, x_2^2x_1, x_2\}$ & $\Phi(U_5'')\!=\bar{\Phi}(U_5'')\!= \{x_2x_1^3, x_2x_1x_2x_1^2, x_2^2x_1^2, x_2^2x_1x_2x_1,x_2^3x_1\}$\\ \hline
\end{tabular}\end{center}}

Case $(4)$ shows that $\Phi\neq \bar{\Phi}$ in general.
\end{example}

\begin{example}[Fibonacci words]\label{Example-Fibonacci}
The sequence of \emph{Fibonacci words} $\{f_m\}_{m\geq0}$ on two words $x_1,x_2$ is given by the initial conditions $f_0=x_1$, $f_1=x_2$, and then for $r\geq1$, $f_{2r} = f_{2r-1}f_{2r-2}$ and $f_{2r+1} = f_{2r-1}f_{2r}$. Obviously the length of the $p$-th Fibonacci word is the $p$-th Fibonacci number. A simple induction on $r$ gives that every Fibonacci word is a Lyndon word and they are sorted as
$$
f_0 <_{\rm lex} f_2 <_{\rm lex} \cdots <_{\rm lex} f_{2r} <_{\rm lex} \cdots <_{\rm lex} f_{2r+1} <_{\rm lex} \cdots <_{\rm lex} f_3 <_{\rm lex} f_1.
$$
Also one has the following facts:
\begin{enumerate}
\item$ \sh(f_{2r})=(f_{2r-1},f_{2r-2})$ and $\sh(f_{2r+1}) = (f_{2r-1}, f_{2r})$ for any $r\geq1$.
\item The set $U$ of all Fibonacci words is a closed set with $\Phi(U)=\{f_{2r-1}f_{2r+1}\,|\, r\geq1\} \cup \{f_{2r}f_{2r-2}\,|\, r\geq 1\}.$
\item For any $p\geq 2$, the set $U_p=\{f_0,f_1,\cdots, f_{p-1}\}$ is a closed set with
    $
    \Phi(U_p)=\{f_{2r-1}f_{2r+1}\,|\, r\geq1,\, 2r+1\leq p\} \cup \{f_{2r}f_{2r-2}\,|\, r\geq 1,\,2r\leq p\} \cup\{f_p\},
    $
    in particular, $\#(\Phi(U_p))=p-1$.
\end{enumerate}
Clearly, the sets $U_2,U_3,U_4,U_5$ coincide with that given in Example \ref{example-closed-set}.
\end{example}

\begin{proof}
(1) A simple induction on $r$ together with the part (3) of Proposition \ref{proposition-facts-Lyndon} gives the result.

(2) We firstly show that any proper Lyndon factor $v$ of $f_{2r}$ is a factor of $f_{2r-1}$ or of $f_{2r-2}$. By Lemma \ref{Lemma-Shirshov-factor-dichotomy} and the part (1), it suffices to exclude the possibility that $v$ is a prefix of $f_{2r}$ such that $l(v)>l(f_{2r-1})$. Otherwise there would exist $1\leq k\leq r-1$ such that $f_{2r}=f_{2r-2k-1}wf_{2r-2k-1}f_{2r-2k-2}$ and $l(f_{2r-1})\leq l(f_{2r-2k-1}w) < l(v) \leq l(f_{2r-2k-1}wf_{2r-2k-1})$. Let $v=f_{2r-2k-1}wt$, where $t\neq1$. Then $t$ is both a prefix and a suffix of $v$, which is impossible. A similar discussion gives that any proper Lyndon factor $v$ of $f_{2r+1}$ is a factor of $f_{2r-1}$ or of $f_{2r}$. Thus Lyndon factors of Fibonacci words are Fibonacci words and so $U$ is a closed set.

Now set $W=\{f_{2r-1}f_{2r+1}\,|\, r\geq1\} \cup \{f_{2r}f_{2r-2}\,|\, r\geq 1\}$. We have $W\subseteq \Phi(U)$ by Proposition \ref{Propositoin-bound-Phi-construction}. Note that every $f_p$ has a proper prefix $f_{2r-1}$ (resp. a suffix $f_{2r}$) whenever $2r-1<p$ (resp. $2r<p$). Therefore if $i>j+1$, then $f_{2i}f_{2j}$ has a proper suffix $f_{2j+2}f_{2j}$ and $f_{2j-1}f_{2i-1}$ has a proper prefix $f_{2j-1}f_{2j+1}$. Also if $i>j$ then $f_{2i-1}f_{2j-2}$ has a proper suffix $f_{2j}f_{2j-2}$, if $i=j$ then $f_{2j-1}f_{2j-2}=f_{2j}$, if $i=j-1$ then $f_{2j-3}f_{2j-2}=f_{2j-1}$, and if $i<j-1$ then $f_{2i-1}f_{2j-2}$ has a proper prefix $f_{2i-1}f_{2i+1}$. Since every element $u\in\Phi(U)$ is of the form $f_rf_s$ with $f_r>_{\rm lex}f_s$, the above discussion implies that $u\in W$. Thus $\Phi(U)=W$.

(3) Clearly $U_p$ is a closed set by the part (2) above. The equality is obtained by a similar discussion of that for $U$ given in the previous paragraph.
\end{proof}

The next result gives a bound of $\Phi(U)$ and $\bar{\Phi}(U)$ for any $U\in \mathcal{Y}$. It generalizes \cite[Lemma 5.2]{GF}. For any closed set $U$, let $$\Upsilon(U)=\{(u,v)\,|\, u,v\in U,\, u>_{\rm lex}v\, \text{and there is no}\, w \in U \,\text{such that}\, u>_{\rm lex}w>_{\rm lex}v \}.$$

\begin{proposition}\label{Propositoin-bound-Phi-construction}
Let $U$ be a closed set of Lyndon words. Then $\sh(uv)=(u,v)$ for any $(u,v)\in \Upsilon(U)$ and the map $\Upsilon(U) \to \mathbb{L}$, given by assigning each $(u,v)\in \Upsilon(U)$ to the concatenating $uv$, is an injection with image contained in $\Phi(U)$. In particular, $$\{uv\,|\,(u,v)\in \Upsilon(U)\}\subseteq \Phi(U)\subseteq \bar{\Phi}(U)\subseteq \{uv\,|\,u>_{\rm lex}v\in U\}\backslash U.$$
\end{proposition}

\begin{proof}
Firstly we show that $\sh(uv)=(u,v)$ for any $(u,v)\in \Upsilon(U)$. If $u$ is a letter, then it is true. If $u$ is not a letter, write $\sh(u)=(u',u'')$. Then $u>_{\rm lex}v\geq_{\rm lex} u''$ and thus $\sh(u)=(u,v)$ by the part (3) of Proposition \ref{proposition-facts-Lyndon}.

Secondly we show that the map is injective. Otherwise, assume that $uv=u'v'$ for some distinct pairs $(u,v),\,(u'v')\in \Upsilon(U)$, then without lost of generality, let $u$ be a proper prefix of $u'$. Observe that there is a factorization $u'=rw,\, v=ws$ with $r,w,s\neq1$, one gets a contradiction $u>_{\rm lex} u' >_{\rm lex} w >_{\rm lex} v$ by Proposition \ref{Proposition-character-Lyndon}.

Thirdly we show $uv\in \Phi(U)$ for any $(u,v)\in \Upsilon(U)$. Note that $uv\not\in U$ for otherwise one would have $u>_{\rm lex}uv>_{\rm lex}v$. Assume that $uv\not\in \Phi(U)$, then $uv$ has a proper factor $w\in \Phi(U)$. Let $\sh(w)=(w',w'')$, then $w',w''\in U$ and $w$ is a prefix of $uv$ with $l(w')\geq l(u)$ by Lemma \ref{Lemma-Shirshov-factor-dichotomy}. If $l(w')=l(u)$, then $w''$ is a proper prefix of $v$ and hence $u=w'>_{\rm lex} w'' >_{\rm lex} v$ which is impossible. If $l(w')>l(u)$, then there are factorizations $w'=us,\, v=sv''$ with $s,v''\neq1$, which gives rise to a contradiction $u>_{\rm lex} w' >_{\rm lex} s >_{\rm lex} v$ by Proposition \ref{Proposition-character-Lyndon}. The inclusion relations now follows immediately.
\end{proof}

\begin{corollary}
Let $U$ be a finite closed set of Lyndon words with $\#(U)=d$. Then
$$d-1 \leq \#(\Phi(U)) \leq \#(\bar{\Phi}(U)) \leq \frac{d(d-1)}{2}.$$
\end{corollary}

\begin{definition}
Let $V$ be an antichain of words. The {\it graph of chains on $V$} is the directed graph $\Gamma(V)$, whose set of vertices is $\{1\} \cup (X\backslash V) \cup S$, where $S$ is the set of all proper suffix of words in $V$ that is of length $\geq2$, and whose set of arrows is defined as follows: for any $x_i\in X\backslash V$ there is an arrow $1\to x_i$; for any two vertices  $u,v\in (X\backslash V) \cup S$, there is an arrow $u\to v$ if and only $uv$ has a unique factor in $V$ which is a suffix of $uv$.

Clearly if $1\to v_1\to v_2\to \cdots \to v_p$ and $1\to v_1'\to v_2'\to \cdots\to v_q'$ are two paths in $\Gamma(V)$ such that $v_1v_2\cdots v_p= v_1'v_2'\cdots v_q'$, then $p=q$ and $v_i=v_i'$ for $i=1,\cdots, p$.

For $p\geq 1$, a word $w$ is called a {\it $p$-chain on $V$} if there is path $1\to v_1\to v_2\to \cdots\to v_p$ in $\Gamma(V)$ such that $w=v_1v_2\cdots v_p$. The empty word $1$ is called the $0$-chain on $V$.  Denote by $C_p(V)$ the set of all $p$-chains on $V$ for all $p\geq0$. Readily one has $C_{0}(V)=\{1\}$, $C_{1}(V)=X\backslash V$ and $C_{2}(V)=V\backslash X$.
\end{definition}

\begin{lemma}\label{Lemma-chain-extending}
Let $U$ be a closed set of Lyndon words and let $V=\Phi(U)$. Then
\begin{enumerate}
\item for any path $v_1\to v_2\to v_3$ in $\Gamma(V)$, if there exist $u_1 >_{\rm lex} \bar{u}_2\in U$ and a non-empty suffix $u_1'$ of $u_1$ such that $u_1'\bar{u}_2$ is a suffix of $v_1v_2$ and $l(\bar{u}_2)\leq l(v_2) \leq l(u_1')+l(\bar{u}_2)$, then there exist $u_2,\bar{u}_3\in U$ and a word $t_2$ such that $v_3=t_2\bar{u}_3$, $u_2$ is a suffix of $v_2t_2$, $u_2':=\bar{u}_2t_2$ is a suffix of $u_2$ and $u_1>_{\rm lex}u_2>_{\rm lex}\bar{u}_3$.
\item for any $p\geq 2$ and any path $v_1\to v_2\to \cdots\to v_p$ in $\Gamma(V)$, if there exist $u_1,\bar{u}_2\in U$ such that $u_1\bar{u}_2 = v_1v_2$ and $l(\bar{u}_2)\leq l(v_2)$, then there exist $u_2,\cdots,u_{p-1},\bar{u}_p$ in $U$ and a non-empty suffix $u_i'$ of $u_i$ for $i=2,\cdots, p-1$ such that $u_1>_{\rm lex}u_2>_{\rm lex}\cdots>_{\rm lex}\bar{u}_p$ and $v_1v_2\cdots v_p=u_1u_2'u_3'\cdots u_{p-1}'\bar{u}_p$.
\item for any $p\geq 2$ and any path $1\to v_1\to v_2\to \cdots \to v_p$ in $\Gamma(V)$, there exists a sequence $u_1>_{\rm lex}u_2 >_{\rm lex}\cdots >_{\rm lex}u_p$ in $U$ and a non-empty suffix $u_i'$ of $u_i$ for $i=2,\cdots, p-1$ such that $v_1v_2\cdots v_p=u_1u_2'u_3'\cdots u_{p-1}'u_p$.
\end{enumerate}
\end{lemma}

\begin{proof}
(1) Let $w_2$ be the unique factor of $v_2v_3$ in $V$ with $\sh(w_2)=(w_2',w_2'')$. There are four cases:
     \begin{itemize}
     \item[(i)] $l(w_2'')\leq l(v_3)$ and $l(w_2) \leq l(\bar{u}_2)+ l(v_3)$. Let $v_3=tw_2''$. Set $u_2=\bar{u}_2t_2$ and $\bar{u}_3=w_2''$. Then $u_2$ is a Lyndon word by the part (1) of Proposition \ref{proposition-facts-Lyndon}, and $u_1>_{\rm lex} \bar{u}_2\geq_{\rm lex} u_2\geq_{\rm lex} w_2'>_{\rm lex} w_2''=\bar{u}_3$.
     \item[(ii)] $l(w_2'')\leq l(v_3)$ and $l(w_2) > l(\bar{u}_2)+ l(v_3)$. Then $w_2'=a\bar{u}_2t_2$ and $v_3=t_2w_2''$ for some words $a \ (\neq 1)$ and $t_2$. Set $u_2=w_2'$, $\bar{u}_3=w_2''$. Then $u_1 >_{\rm lex} u_2>_{\rm lex} \bar{u}_3$.
     \item[(iii)] $l(v_3) < l(w_2'')\leq l(\bar{u}_2)+ l(v_3)$. Clearly one can find a Lyndon suffix $\bar{w}$ of $w_2''$ such that $l(\bar{w})>l(v_3)$ and $l(\bar{w}'')\leq l(v_3)$, where $\sh(\bar{w}) = (\bar{w}',\bar{w}'')$. Let $v_3=t_2\bar{w}''$. Set $u_2= \bar{u}_2t_2$ and $\bar{u}_3=\bar{w}''$. Then again $u_2$ is a Lyndon word by the part (1) of Proposition \ref{proposition-facts-Lyndon}, and $u_1 >_{\rm lex} \bar{u}_2 \geq_{\rm lex} u_2 \geq_{\rm lex} \bar{w}' >_{\rm lex} \bar{w}''=\bar{u}_3$.
     \item[(iv)] $l(w_2'')> l(\bar{u}_2)+ l(v_3)$. Then one can find a Lyndon suffix $\bar{w}$ of $w_2''$ such that $l(\bar{w})>l(\bar{u}_2)+ l(v_3)$ and $l(\bar{w}'')\leq l(\bar{u}_2)+ l(v_3)$, where $\sh(\bar{w}) = (\bar{w}',\bar{w}'')$. In this situation, there are two subcases: one is $l(\bar{w}'')\leq l(v_3)$, which reduces to the case (ii); and the other one is $l(\bar{w}'')>l(v_3)$, which reduces to the case (iii).
     \end{itemize}
    We have exhausted all possibilities and the result (1) follows.

(2) Let $u_1'=u_1$. Apply the part (1) iteratively $p-2$ times, one gets the result.

(3) Let $u_1, \bar{u}_2$ be given by $\sh(v_1v_2)=(u_1,\bar{u}_2)$. Readily $l(\bar{u}_2)\leq l(v_2)$ because $v_1$ is a letter. The result now follows from the part (2) by setting $u_p=\bar{u}_p$.
\end{proof}

\begin{lemma}\label{Lemma-chain-Lyndon}
Let $U$ be a closed set of Lyndon words and let $V=\Phi(U)$. Let $p\geq 2$ and $u_1,u_2,\cdots,u_p\in U$ with $(u_1,u_2),\cdots,(u_{p-1},u_p)\in \Upsilon(U)$. Let $u_1=x_iu_1'$. Then $1\to x_i\to u_1'u_2\to u_3\to \cdots \to u_p$ is a path in $\Gamma(V)$. In particular the concatenating $u_1u_2\cdots u_p$ is a $p$-chain on $V$.
\end{lemma}

\begin{proof}
Note that $x_i,u_1'u_2, u_3,\cdots,u_p$ are vertices of $\Gamma(V)$ and there are arrows $x_i\to u_1'u_2$, $u_3\to u_4$,$\cdots$, $u_{p-1}\to u_p$ in $\Gamma(V)$ by Proposition \ref{Propositoin-bound-Phi-construction}. Thus it remains to show that there is an arrow $u_1'u_2\to u_3$ in $\Gamma(V)$. Otherwise, there is a non-empty suffix $l$ of $u'$ and a non-empty proper prefix $r$ of $u_3$ such that $w=lu_2r\in V$. Then by Lemma \ref{Lemma-Shirshov-factor-dichotomy}, there are four cases:
\begin{itemize}
\item[(i)] $\sh(w)=(l,u_2r)$. One has $u_2>_{\rm lex} u_2r>_{\rm lex}u_3$.
\item[(ii)] $\sh(w)=(l',l''u_2r)$ for some non-empty proper prefix $l'$ of $l$. One has $u_1>_{\rm lex} l''u_2r >_{\rm lex} u_3$.
\item[(iii)]  $\sh(w)=(lu_2, r)$. One has $u_1 >_{\rm lex} lu_2 >_{\rm lex} u_2$.
\item[(iv)] $\sh(w)=(lu_2r', r'')$ for some non-empty proper suffix $r''$ of $r$. One has $u_1 >_{\rm lex} lu_2r' >_{\rm lex} u_3$.
\end{itemize}
Since there is no member in $U$ other than $u_2$ between $u_1$ and $u_3$, all cases listed above is impossible and hence there is an arrow $u_1'u_2\to u_3$ in $\Gamma(V)$.
\end{proof}

\begin{remark}
In \cite[Section 4]{GI}, the author characterized the graph of chains on an antichain $V$, for which the monomial algebra $k\langle X\rangle/(V)$ is of finite global dimension and of finite GK-dimension. It is worth to mention that Lemma \ref{Lemma-chain-Lyndon} is equivalent to \cite[Lemma 4.21]{GI}, while the part (3) of Lemma \ref{Lemma-chain-extending} is stronger than \cite[Lemma 4.22]{GI}.
\end{remark}

\begin{proposition}\label{proposition-chain-unique}
Let $U=\{z_1 >_{\rm lex} z_2 >_{\rm lex} \cdots >_{\rm lex} z_d\}$ be a finite closed set of Lyndon words and let $V=\Phi(U)$. Then $1\to z_1\to z_2 \to \cdots \to z_d$ is the unique path of length $d$ in $\Gamma(V)$ started at $1$ and there is no path of length $p > d$ in $\Gamma(V)$ started at $1$. In particular, $C_d(V)=\{z_1z_2\cdots z_d\}$ and $C_p(V)=\emptyset$ for any $p>d$.
\end{proposition}

\begin{proof}
Note that $z_1=x_1$ and so $1\to z_1\to z_2\to \cdots\to z_d$ is a path in $\Gamma(V)$ by Lemma \ref{Lemma-chain-Lyndon}. Also there are no paths of length $p>d$ in $\Gamma(V)$ started at $1$ by the part (3) of Lemma \ref{Lemma-chain-extending}. Assume that $1\to v_1\to v_2\to \cdots \to v_d$ is a path in $\Gamma(V)$. Write $\sh(v_1v_2)=(u_1,\bar{u}_2)$. Then $u_1=z_1$ by the part (2) of Lemma \ref{Lemma-chain-extending} and so it is a letter. Hence $v_1=z_1$. Apply the part (1) of Lemma \ref{Lemma-chain-extending} on the path $z_1\to v_2\to v_3$ with $u_1'=u_1=z_1$ and $\bar{u}_2=v_2$, one gets $u_2=v_2t_2$ and $u_2\bar{u}_3=v_2v_3$. Thus $u_2=z_2$ by applying the part (2) of Lemma \ref{Lemma-chain-extending} on the path $v_2\to v_3\to \cdots \to v_p$. Since $z_2\geq_{\rm lex} v_2$, one has $v_2=z_2$. Now by induction suppose $v_j=z_j$ for $j\leq i$, where $i\geq2$. Then apply the part (1) of Lemma \ref{Lemma-chain-extending} on the path $z_{i-1}\to z_i \to v_{i+1}$ with $u_{i-1}'=u_{i-1}=z_{i-1}$ and $\bar{u}_i=z_i$, then $u_i >_{\rm lex} \bar{u}_{i+1}$, $v_{i+1} = t_i\bar{u}_{i+1}$ and $u_i'=u_i=z_it_i$.  If $t_i\neq 1$, the part (2) of Lemma \ref{Lemma-chain-extending} for the path $z_i\to v_{i+1}\to \cdots\to v_p$ gives that $z_i>_{\rm lex}u_i >_{\rm lex} u_{i+1}>_{\rm lex}\cdots >_{\rm lex} u_d$ and thus there are $d-i+1$ elements in $U$ that is less than $z_i$, which is impossible. So $t_i=1$ and hence $z_i=u_i, v_{i+1}=\bar{u}_{i+1}$. Note that $\bar{u}_{i+1}$ is a prefix of $u_{i+1}$ in this case and hence $z_i>_{\rm lex} \bar{u}_{i+1} \geq_{\rm lex} u_{i+1} >_{\rm lex} u_{i+2}>_{\rm lex}\cdots >_{\rm lex} u_d$. Therefore, one has  $v_{i+1} = \bar{u}_{i+1} = u_{i+1} = z_{i+1}$. The last statement is obtained by definition.
\end{proof}

Now we turn to the main result of this section, which provides an efficient way to calculate invariants of certain graded algebras.

\begin{theorem}\label{Theorem-invariant-algebra}
Let $U$ be a closed set of Lyndon words and let $A=k\langle X\rangle/(G)$, where $G$ is a set of homogeneous polynomials such that $\Phi(U)\subseteq \lw(G) \subseteq \mathbb{L}\backslash U$. If $G$ is a Gr\"{o}bner set then one has:
\begin{enumerate}
\item The set of obstructions of $A$ is $\Phi(U)$.
\item ${\rm Irr}(G) = \{u_1u_2\cdots u_s\,|\,u_1\leq_{\rm lex} u_2\leq_{\rm lex} \cdots\leq_{\rm lex} u_s \in U,\ s\geq1\}\cup \{1\}$.
\item
The Hilbert series of $A$ is $H_A(t)=\prod_{u\in U}\big(1-t^{\deg(u)}\big)^{-1}$ and $\gkdim A=\#(U)$.
\item The global dimension of $A$ satisfies $\gldim(A)\leq \#(\Phi(U))+1$.
\item If $U$ is a finite set then $\gldim A = \#(U)$ and the last nonzero term of the minimal free resolution of $k_A$ is $A(-l)$, where $l=\sum_{u\in U}\deg(u)$.
\end{enumerate}
\end{theorem}

\begin{proof}
(1)
By the definition, the set of obstructions of $A$ coincides with the set $$\{u\in \lw(G)\,|\,u \text{ has no proper factor in }\lw(G)\},$$
which equals to $\Phi(U)$.

(2)
Note that $U=\Psi(\lw(G))$. Apply the parts $(4),(5)$ of Proposition \ref{proposition-facts-Lyndon}, one gets the result.

(3)
It is an easy combinatoric exercise to see $H_A(t)=\prod_{u\in U}\big(1-t^{\deg(u)}\big)^{-1}$. Then by \cite[Corollary 2.2]{SZ} one obtains $\gkdim A=\#(U)$.

(4)
Assume that $\Phi(U)$ is a finite set. Suppose $1\to v_1\to v_2\to \cdots \to v_p$ is a path in the graph $\Gamma(\Phi(U))$. Let $w_i$ be the unique factor of $v_iv_{i+1}$ in $\Phi(U)$ for $i=1,\cdots, p-1$. Immediately one has $w_1>_{\rm lex}w_2>_{\rm lex}\cdots >_{\rm lex}w_{p-1}$. Thus $p-1\leq \#(\Phi(U))$. Therefore there is no $p$-chain on $\Phi(U)$ if $p > \#(\Phi(U))+1$. Thus \cite[Theorem 1.5]{An2} gives that $\gldim (A) =\text{pd}(k_A) \leq \#(\Phi(U))+1$.

(5)
Let $\#(U)=d$ and let $0\to F_d\to F_{d-1} \to \cdots \to F_1\to A\to k\to 0$ be the minimal free resolution of $k_A$. Consider the Anick's resolution $\mathcal{E}$ of $k_A$:
\begin{eqnarray*}
&&\cdots \to kC_p\otimes A\to \cdots \to kC_d\otimes A\xrightarrow{\delta_d} kC_{d-1}\otimes A \to \cdots \to kC_1\otimes A \to A\to k_A\to 0,
\end{eqnarray*}
where $C_i= C_i(\Phi(U))$. By Proposition \ref{proposition-chain-unique}, one has $kC_p=0$ for any $p>d$ and $kC_d \cong k(-l)$, where $l=\sum_{u\in U}\deg(u)$. Also by the part (3) of Lemma \ref{Lemma-chain-extending}, one has $kC_{d-1} \cong \bigoplus_{c\in C_{d-1}}k(-r_c)$ with $r_c < l$ for each $c\in C_{d-1}$. Thus in the complex $\mathcal{E}\otimes_A k$, the differential $\delta_d\otimes_A k=0$ and hence $\tor^A_d(k,k)=k(-l)$ and $\tor^A_p(k,k)=0$ for any $p>d$. The result now follows from the facts $\gldim A = \sup \{ n\,|\, \tor^A_n(k_A,{}_Ak)\neq0\}$ and $F_p\cong \tor^A_{p}(k_A,{}_Ak)\otimes A$.
\end{proof}

\begin{lemma}\label{Lemma-bound-Gorenstein-parameter}
Let $U$ be a finite closed set of Lyndon words on $X=\{x_1,\cdots,x_n\}$ with $\#(U)=d$. Assume that $\deg(x_1)=\cdots=\deg(x_n)=1$, then
$$
\sum_{u\in U} \deg(u) \leq \frac{\phi^{d-n+4}-\psi^{d-n+4}}{\phi-\psi}+n-3, \quad \text{ where }\, \phi=\frac{1+\sqrt{5}}{2},\, \psi=\frac{1-\sqrt{5}}{2}.
$$
\end{lemma}

\begin{proof}
Denote $a_i$ the $i$-th Fibonacci number and  $s_i=\sum_{j=0}^ia_j$. Then, by induction, $s_i=a_{i+2}-1$. Enumerating elements in $U$ in deglex order as
$$
x_1 \leq_{\rm deglex} \cdots \leq_{\rm deglex} x_{n-2} \leq_{\rm deglex} x_{n-1}=g_0\leq_{\rm deglex} x_n=g_1\leq_{\rm deglex} g_2 \leq_{\rm deglex} \cdots \leq_{\rm deglex} g_{d-n+1},
$$
then $\deg(g_i)\leq a_i$ for $i=0,\cdots,d-n+1$ as showed in \cite[Proposition 7,3]{GF}. Therefore
$$\sum_{u\in U} \deg(u) \leq s_{d-n+1} + n-2 = a_{d-n+3}+n-3 = \frac{\phi^{d-n+4}-\psi^{d-n+4}}{\phi-\psi}+n-3,$$
where the last equality follows from the general formula $a_i=\frac{\phi^{i+1}-\psi^{i+1}}{\phi-\psi}$ for $i\geq0$.
\end{proof}

Combine the parts (3), (5) of Theorem \ref{Theorem-invariant-algebra} and Lemma \ref{Lemma-bound-Gorenstein-parameter}, we get:

\begin{corollary}\label{Corollary-bound-Gorenstein-parameter}
Let $A$ be an AS-regular algebra of global dimension $d$ and of Gorenstein parameter $l$ which is generated in degree one with $\dim A_1=n$. If all obstructions of $A$ with respect to some choice of a basis $x_1,\cdots,x_n$ of $A_1$ are Lyndon words, then $\gkdim A=d$ and $l\leq \frac{\phi^{d-n+4}-\psi^{d-n+4}}{\phi-\psi}+n-3.$
\end{corollary}

\begin{remark}
Note that AS-regular algebras of which the obstructions are Lyndon words is a context more general than the universal enveloping algebras of positively graded Lie algebras. Thus Corollary \ref{Corollary-bound-Gorenstein-parameter} can be thought of as a counterpart of \cite[Proposition 2.4]{FV}. However, we are unable to decide whether or not the given estimation is optimal. Indeed, suppose that $\mathfrak{g}$ is a graded Lie algebra of dimension $d$ which is generated in degree one with $\dim \mathfrak{g}_1=n$. Then alike the discussion in \cite[Proposition 2.4]{FV}, the highest possible Gorenstein parameter of the universal enveloping algebra $U(\mathfrak{g})$ is
\begin{align*}
\binom{d-n+2}{2}+n-1 &= (1+1+2+\cdots+d-n+1) + n-2\\
                     &\leq(a_0+a_1+a_2+\cdots+a_{d-n+1})+n-2\\
                     &=\frac{\phi^{d-n+4}-\psi^{d-n+4}}{\phi-\psi}+n-3.
\end{align*}
It is interesting to find AS-regular algebras which is generated in degree one and of which the Gorenstein parameter is larger than $\binom{d-n+2}{2}+n-1$. Note that the only examples known to us are \cite[Example 4.5]{ZL}.
\end{remark}

\section{Construction of Artin-Schelter regular algebras from Lyndon words}

In this section we turn to the construction of AS-regular algebras from Lyndon words. The major tool is the generalized bracketing on Lyndon words defined inductively by Shirshov factorization. Throughout this section, $X=\{x_1,\cdots, x_n\}$, $q=[q_{i,j}]_{n\times n}$ is a matrix with nonzero entries in $k$. Also $u,v,w$ denote words on $X$.

Let $q_{u,1}=q_{1,u}=1$, then define $q_{x_i,x_ju}=q_{i,j}q_{x_i,u}$ and $q_{ux_i,v}=q_{u,v}q_{x_i,v}$ inductively. Obviously, the following equalities hold:
\begin{equation*}
q_{u,vw}=q_{u,v}q_{u,w},\quad q_{uv,w}=q_{u,v}q_{v,w}.
\end{equation*}
Let $[-,-]: k\langle X\rangle \otimes k\langle X\rangle\to k\langle X\rangle$ be the bilinear operation given by $[u,v]=uv-q_{u,v}vu$. Then
\begin{align*}
[[u,v],w] &= [u,[v,w]]-q_{u,v} v[u,w]+q_{v,w}[u,w]v, \\
[uv,w] &= u[v,w] + q_{v,w}[u,w]v,  \\
[u,vw] &= [u,v]w + q_{u,v} v[u,w].
\end{align*}
Now set $[x_i]=x_i$ for letters $x_i\in X$. If $u$ is a Lyndon word with $l(u)>1$ and $\sh(u)=(u',u'')$, define
\begin{equation*}
[u] := [[u'],[u'']].
\end{equation*}
Note that, inductively, $[u]$ is a monic homogeneous polynomial with $\lw([u])=u$ and all words occur in $[u]$ have the same constitute. In particular $[u]=[u'][u'']-q_{u',u''}[u''][u']$ if $\sh(u)=(u',u'')$.

\begin{definition}
A polynomial of the form $[u]$, where $u$ is a Lyndon word, is called a \emph{super-letter}. A finite product of super-letters is called a \emph{super-word}.
\end{definition}

Obviously, if $D=[u_1]\cdots[u_r]$ is a super-word, then $D$ is a monic homogeneous polynomial with $\lw(D)=u_1\cdots u_r$ and every word occurs in $D$ has the same constitute with $u_1u_2\cdots u_r$. Moreover, super-letters are in one-to-one correspondence with Lyndon words, and every super-word has a unique factorization in super-letters \cite[Lemma 2.5]{He}. Thus one can define a total order on the set of all super-letters by
$$
[u]> [v]\;\; \text{ iff }\;\; u>_{\rm lex}v,
$$
and it then extends to a lexicographical order on the set of super-words in a natural way.

\begin{definition}
A super-word $D=[u_1][u_2]\cdots [u_r]$ is called \emph{monotonic} if $[u_1]\leq [u_2]\leq \cdots \leq [u_r]$.
\end{definition}

It is easy to see that all monotonic super-words form a basis of $k\langle X\rangle$ by the part (4) of Proposition \ref{proposition-facts-Lyndon}. Moreover, if $D=[u_1]\cdots[r_r]$ and $D'=[u'_1]\cdots[u'_s]$ are monotonic super-words, then $D>_{\rm lex}D'$ iff $u_1\cdots u_r >_{\rm lex} u'_1\cdots u'_s$ \cite[Lemma 5]{Kh}. Thus
with the part (4) of Proposition \ref{proposition-facts-Lyndon} at hand, the lexicographical order on super-words is compatible with that on ordinary words.

In the sequel, we will associate to each tuple $I=(i_1,\cdots,i_r)$ a new variable $x_I=x_{(i_1,\cdots,i_r)}$. We will frequently write $x_I=x_{i_1i_2\cdots i_r}$ if there is no risk of confusions. In particular, we identify $x_{(i)}=x_i$ for $i=1, \cdots, n$. For any Lyndon word $u=x_{i_1}\cdots x_{i_r}$ let $\rho(u)=(i_1,\cdots, i_r)$, and for any closed set $U$ of Lyndon words let $$
X_U=\{x_{\rho(u)}\,|\,u\in U\}.
$$
The degree function on $X_U^*$ is given by assigning $x_{\rho(u)}$ to the number $\deg(u)$. Clearly $X\subseteq X_U$ and $k\langle X\rangle$ is a graded subalgebra of $k\langle X_U\rangle$. The total order on $X$ extends to a total order on $X_U$ by saying $x_{\rho(u)}> x_{\rho(v)}$ iff $u>_{\rm lex}v$ for any $u,v\in U$. Thus the lex order and the deglex order on $X_U^*$ are extensions of that on $X^*$ respectively.

Define $[-,-]:k\langle X_U\rangle \otimes k\langle X_U\rangle \to k\langle X_U\rangle$ to be the linear map by
$$
[D,D']=DD'-q_{u_1\cdots u_r, v_1\cdots v_s}D'D,
$$
where $D=x_{\rho(u_1)}\cdots x_{\rho(u_r)},\, D'=x_{\rho(v_1)}\cdots x_{\rho(v_s)}$. It is an extending of the linear map $[-,-]:k\langle X\rangle \otimes k\langle X\rangle \to k\langle X\rangle$. Let
$
\beta=\beta_U:k\langle X_U\rangle \to k\langle X\rangle\
$
be the homomorphism of graded algebras given by $\beta(x_{\rho(u)})=[u]$ for each $u\in U$. Note that $
\beta|_{k\langle X\rangle}=\id_{k\langle X\rangle}.
$

For any $u\in U\backslash X$ with $\sh(u)=(v,w)$, we denote
$$
f_u=[x_{\rho(v)},x_{\rho(w)}] - x_{\rho(u)}.
$$
Then a simple induction on the length of $u\in U$ implies that
$
[u]-x_{\rho(u)}\in (F_U), \text{ where } F_U=\{f_u\,|\, u\in U\backslash X\}.
$
Therefore, for any set $G$ of homogeneous polynomials, $\beta$ induces an isomorphism of graded algebras
$$
k\langle X_U\rangle/(F_U\cup G) \xrightarrow{\cong} k\langle X\rangle/(G).
$$
The inverse is the homomorphism induced by the injection map $k\langle X\rangle \subseteq k\langle X_U\rangle$. In particular, when $G=\emptyset$, one gets
$k\langle X_U\rangle/(F_U)\cong k\langle X\rangle$.

\begin{lemma}\label{Lemma-calcultion}
Let $U$ be a closed set of Lyndon words. For each $v\in \bar{\Phi}(U)$ we associate with a homogeneous polynomial $g_v\in k\langle X\rangle$ of degree
$\deg(v)$ that is a linear combination of super-words $[u_1]\cdots [u_r]<_{\rm lex}[v]$ with $u_1,\cdots, u_r\in U$.
Let $G=\{\overline{g}_v=[v]-g_v\,|\,v\in \bar{\Phi}(U)\}$. Then for each pair $u>_{\rm lex}u'\in U$, there is a homogeneous polynomial $h_{u,u'}\in k\langle X_U\rangle$ of degree $\deg(uu')$ such that
\begin{enumerate}
\item
$h_{u,u'}$ is a linear combination of lexicographical lesser words than $x_{\rho(u)}x_{\rho(u')}$,
\item $(H)=(F_U\cup G)$, where $H=\{\bar{h}_{u,u'}=[x_{\rho(u)},x_{\rho(u')}]-h_{u,u'}\,|\,u>_{\rm lex}u'\in U\}$.
\end{enumerate}
Moreover, $k\langle X\rangle/(G)\cong k\langle X_U\rangle/(H)$ as graded algebras and the following statements are equivalent:
\begin{enumerate}
\item[(i)] $G$ is a Gr\"{o}bner set with respect to the deglex order on $X^*$.
\item[(ii)] $H$ is a Gr\"{o}bner set with respect to the deglex order on $X_U^*$.
\item[(iii)] $J(u,v,w) = [h_{u,v},x_{\rho(w)}] - [x_{\rho(u)}, h_{v,w}] + q_{u,v} x_{\rho(v)}h_{u,w} - q_{v,w} h_{u,w}x_{\rho(v)}$ is trivial modulo $H$ with respect to the deglex order on $X_U^*$ for each triple $u>_{\rm lex}v>_{\rm lex}w \in U$.
\end{enumerate}
\end{lemma}

\begin{proof}
By a simpler discussion as that given in the proof of \cite[Lemma 3.6]{He}, there is a homogeneous polynomials $c_{(u|u')}\in k\langle X\rangle$ of degree $\deg(uu')$ for each pair $u>_{\rm lex}u'\in U$ such that
\begin{itemize}
\item $\bar{c}_{(u|u')}=[[u],[u']]-c_{(u|u')} \in (G)$,
\item $c_{(u|u')}$ is a linear combination of super-words $[u_1]\cdots [u_r]\leq_{\rm lex}[uu']$ with $u_1,\cdots,u_r\in U$.
\end{itemize}
In particular, if $\sh(uu')=(u,u')$ and $uu'\in U$ then $c_{(u|u')}=[uu']$, and if  $\sh(uu')=(u,u')$ but $uu'\not\in U$ then $c_{(u|u')}=g_{uu'}$.  Let $h_{u,u'}$ be the polynomial in $k\langle X_U\rangle$ obtained from $c_{(u|u')}$ by replacing each super-word $[u_1]\cdots [u_r]$ occurs in $c_{(u|u')}$ with $x_{\rho(u_1)}\cdots x_{\rho(u_r)}$. Clearly $h_{u,u'}$ is a linear combination of lexicographical lesser words than $x_{\rho(u)}x_{\rho(u')}$. By construction, we have the following observations:
\begin{itemize}
\item $f_u=\bar{h}_{u',u''}$ for any word $u\in U\backslash X$ with $\sh(u)=(u',u'')$ and so $F_U\subseteq H$,
\item $\beta(\bar{h}_{u,u'}) = \bar{c}_{(u|u')}$ and so $\bar{h}_{u,u'}-\bar{c}_{(u|u')} \in (F_U)$.
\end{itemize}
Hence one has $(F_U\cup G) = (F_U \cup \{\bar{c}_{(u|u')}\,|\,u>_{\rm lex}u'\in U\}) = (H)$.
Note that $\beta$ induces an isomorphism of graded algebras $k\langle X_U\rangle/(H) \xrightarrow{\cong} k\langle X\rangle/(G)$. Thus all monotonic super-words in super-letters $[u], u\in U$ form a basis of $k\langle X\rangle/(G)$ iff all monotonic words on $X_U$ form a basis of $k\langle X_U\rangle/(H)$. Then Proposition \ref{Proposition-Diamond-Lemma} implies the final equivalence relations.
\end{proof}

\begin{remark}
Lemma \ref{Lemma-calcultion} is used in the sequel to simplify the computation in determining whether $G$ is a Gr\"{o}bner set (see Example \ref{Example-AS-regular}). However, it is interesting in its own right. Given a Lie algebra $\mathfrak{g}=\text{Lie}(X)/(G)_L$, where $G$ is a set of Lie polynomials and $(G)_L$ is the Lie ideal of $\text{Lie}(X)$ generated by $G$. Assume that $\{e_i\}_{i\in I}$ is an ordered basis of $\mathfrak{g}$ with structure constants $\{[e_i,e_j]-\sum a_{i,j}^re_r\}_{i>j}$, then Lemma \ref{Lemma-calcultion} is a generalization of the following formula:
$$
k\langle X\rangle/(G) \cong U(\mathfrak{g}) \cong k\langle x_i\;|\;i\in I\rangle/(\{[x_i,x_j]- \textstyle \sum a_{i,j}^r x_r\}_{i>j}),
$$
where $U(g)$ denotes the universal enveloping algebra of $\mathfrak{g}$.
\end{remark}

\begin{definition}
Let $G \subseteq k\langle X\rangle$ be a set of homogeneous polynomials. A super-letter $[u]$ is said to be \emph{hard modulo $G$} if, in the algebra $k\langle X \rangle/(G)$, $[u]$ is not a linear combination of super-words of the same degree in lexicographical lesser super-letters than $[u]$.

Clearly if $u$ is a Lyndon word that is irreducible modulo $G$, then $[u]$ is hard modulo $G$.
\end{definition}

\begin{lemma}\label{Lemma-constructing-AS-regular}
Let $U$ be a finite closed set of Lyndon words and let $A=k\langle X\rangle/(G)$, where $G$ is a set of homogeneous polynomials such that $\Phi(U)\subseteq \lw(G) \subseteq \mathbb{L}\backslash U$. If $G$ is a Gr\"{o}bner set modulo which there are no hard super-letters other than $[u], u\in U$, then one has:
\begin{enumerate}
\item
There exists an algebra filtration $\mathcal{F} = \{F_i\}_{i\geq 0}$ one $A$ such that,
as graded algebras,
$$
\text{\rm gr}_{\mathcal{F}}(A)\, \cong \, k\langle X_U\rangle/([x_{\rho(u)},x_{\rho(u')}]: u>_{\rm lex}u'\in U).
$$
\item
$A$ is AS-regular of global dimension $d=\#(U)$ and of Gorenstein parameter $l=\sum_{u\in U} \deg(u)$. Moreover, $A$ is strongly Noetherian, Auslander-regular and Cohen-Macaulay.
\end{enumerate}
\end{lemma}

\begin{proof}
(1)
By Theorem \ref{Theorem-invariant-algebra}, the set of all monotonic super-words in super-letters $[u],u\in U$ is a basis of $A$. Since there are no hard super-letters other than $[u],u\in U$, if $v$ is a reducible word, then $[v]$ is not hard and hence it is a linear combination of super-words of the same degree in lexicographical lesser super-letters than $[v]$ in $A$. A similar, but simpler, discussion of that given in \cite[Lemma 14, Theorem 3]{Kh} implies the result.

(2) By \cite[Lemma 3.5]{ZZ1}, \cite[Lemma 5.3]{ZZ2} and the part (1), we obtain that $\text{gr}_{\mathcal{F}}(A)$ is AS-regular, strongly Noetherian, Auslander-regular and Cohen-Macaulay, and so is $A$ by \cite[Theorem 3.6]{ZZ1} and \cite[Lemma 4.4]{StZ}. Apply Theorem \ref{Theorem-invariant-algebra}, we have $d=\#(U)$ and $l=\sum_{u\in U} \deg(u)$.
\end{proof}

Note that the sufficient condition given in Lemma \ref{Lemma-constructing-AS-regular} for a graded algebra $k\langle X\rangle/(G)$ to be AS-regular is not easy to check in general. In the remaining of the paper, we focus on providing practical approaches for constructing AS-regular algebras.

\begin{definition}\label{Definition-A(U,q)}
To each closed set $U$ of Lyndon words we associate with a graded algebra
$$
A(U,q)=k\langle X\rangle/(G(U,q)),\quad \text{where}\quad G(U,q)=\{[v]\,|\,v\in \bar{\Phi}(U)\}.
$$

In general, $(G(U,q)) \neq (\{[v]\,|\,v\in\Phi(U)\})$, however, if $G(U,q)$ is a Gr\"{o}bner set, then the obstruction of $A(U,q)$ is $\Phi(U)$ and $(G(U,q))= (\{[v]\,|\,v\in\Phi(U)\})$.
\end{definition}

\begin{theorem}\label{Theorem-approach-1}
Let $U$ be a finite closed set of Lyndon words. If $G(U,q)$ is a Gr\"{o}bner set, then $A(U,q)$ is AS-regular of global dimension $d=\#(U)$ and of Gorenstein parameter $l=\sum_{u\in U} \deg(u)$. Moreover, $A(U,q)$ is strongly Noetherian,  Auslander-regular and Cohen-Macaulay.
\end{theorem}

\begin{proof} Let $A=A(U,q)$.
Note that there are no reducible Lyndon words of length $1$, and if $v$ is a reducible Lyndon word of length $2$ then $[v]=0$ in $A$ because $v\in \bar{\Phi}(U)$. Assume that $[v]=0$ in $A$ for any reducible Lyndon word $v$ of length less than $p$. Now let $v_0$ be an arbitrary reducible Lyndon word of length $p$ with $\sh(v_0)=(v'_0,v''_0)$. If $v_0',v_0''\in U$, then $v_0\in \bar{\Phi}(U)$ and hence $[v_0]=0$ in $A$. If $v_0'\not\in U$ (resp.  $v_0''\not\in U$), then by induction one has $[v_0']=0$ (resp, $[v_0'']=0$) in $A$, and both implies that $[v_0]=0$ in $A$. Therefore $[v]=0$ in $A$ for any reducible Lyndon word $v$ and so there are no hard super-letters other than $[u],u\in U$. Now the result follows from Lemma \ref{Lemma-constructing-AS-regular}.
\end{proof}

\begin{remark}
There is no counterpart of Theorem \ref{Theorem-approach-1} if we set $G(U,q)=\{[v]\,|\,v\in\Phi(U)\}$ in Definition \ref{Definition-A(U,q)}, since in this situation there may exists some $v\in \bar{\Phi}(U)\backslash \Phi(U)$ such that $[v]\not\in (G(U,q))$. However, for any Lyndon words $u,v$ with $u=avb$ there exists a factorization $b=cd$ such that $[u]=[a[vc]d]$. This fact follows immediately from Lemma \ref{Lemma-Shirshov-factor-dichotomy} through a simple induction on the length of $u$. It is crucial in the Gr\"{o}bner bases theory for Lie algebras \cite{BC}.
\end{remark}

\begin{example}
Let $U=X$. Then $G(U,q)=\{\ [x_j,x_i]\;|\;1\leq i<j\leq n\ \}$. Clearly $G(U,q)$ is a Gr\"{o}bner set. The algebra $A(U,q)=\mathcal {O}_q(k^n)$ is called the \emph{multiparameter quantum affine space}. It is AS-regular of global dimension $n$ and of Gorenstein parameter $n$.
\end{example}

\begin{example}\label{Example-AS-regular} We continue to use the symbols given in Example \ref{example-closed-set} and give some AS-regular algebras below by Theorem \ref{Theorem-approach-1}. The computation is based on Lemma \ref{Lemma-calcultion}. In each case, the equation system obtained from all $J(u,v,w)$ with $u>_{\rm lex}v>_{\rm lex}w$ are not very complicated and so it is solved by hand.

(1) $A(U_2,q)=k\langle x_1,x_2\rangle/([x_2,x_1])$. $G(U_2,q)$ is always a Gr\"{o}bner set. $A$ is the quantum plane and it is AS-regular of global dimension $2$ and of Gorenstein parameter $2$.

(2) $A(U_3,q)\cong k\langle x_1,x_{21}, x_2\rangle/(f_1,f_2,f_3)$, where
    \begin{eqnarray*}
    f_1=[x_2,x_{21}],\quad f_2=[x_2, x_1] - x_{21},\quad f_3=[x_{21},x_1].
    \end{eqnarray*}
    $G(U_3,q)$ is a Gr\"{o}bner set iff $q_{22}=q_{11}$. In this case, $A(U_3,q)$ is the algebra of type $S_1$ in \cite{ASc}. It is AS-regular of global dimension $3$ and of Gorenstein parameter $4$.

(3) $A(U_4,q) \cong k\langle x_1,x_{21},x_{221},x_2\rangle /(f_1,f_2,\cdots,f_6)$, where
    \begin{align*}
    &f_1=[x_2,x_{221}],& &f_2=[x_2,x_{21}]-x_{221},& & f_3=[x_2,x_1]-x_{21},\\
    &f_4=[x_{221},x_{21}],& &f_5=[x_{221},x_1]- q_{21}(q_{11}-q_{22})x_{21}^2,& &f_6=[x_{21},x_1].
    \end{align*}
    $G(U_4,q)$ is a Gr\"{o}ner set iff one of the following conditions holds:
    \begin{enumerate}
    \item[(a)] $q_{22}=1,\, q_{11}=1,\, q_{21}q_{12}=1$.
    \item[(b)] $q_{22}=\zeta,\, q_{11}=\zeta,\, q_{21}q_{12}=\zeta^2$, where $\zeta^2+\zeta+1=0$.
    \item[(c)] $q_{11}=q_{22}^2,\, \, q_{21}q_{12}q_{22}^2=1$.
    \end{enumerate}
     Taking Case (a), Case (b) and Case (c), $A(U_4,q)$ are corresponding to $D(-2q_{21},-q_{21})$, $D(q_{21}\zeta^2, -q_{21}\zeta^2)$ and $D(-q_{21}-q_{21}q_{22}^2,-q_{21}q_{22})$ in \cite[Theorem A]{LPWZ} respectively. They are AS-regular of global dimension $4$ and of Gorenstein parameter $7$.

\begin{remark} In the case (c), let $q_{22}$ be a root of the polynomial $x^2-\frac{v}{p}x+1$ and let $q_{21}=pq_{22}-v$, then we obtain the algebra $D(v,p)$; that is, we can reconstruct $D(v,p)$ as a deformation of the universal enveloping algebra $D(-2,-1)$.
\end{remark}

(4) $A(U_5,q) \cong k\langle x_1,x_{21},x_{22121},x_{221},x_2\rangle/(f_1,f_2,\cdots,f_{10})$, where
    \begin{eqnarray*}
    &&f_1=[x_2,x_{221}],\qquad f_2=[x_2,x_{22121}],\qquad f_3=[x_2,x_{21}]-x_{221},\qquad f_4=[x_2,x_1]-x_{21},\\
    &&f_5=[x_{221},x_{22121}],\;\; f_6=[x_{221},x_{21}]-x_{22121},\qquad f_7=[x_{221},x_1]- q_{21}(q_{11}-q_{22})x_{21}^2,\\
    &&
    f_8=[x_{22121},x_{21}], \quad f_9=[x_{22121},x_1]- q_{21}^2q_{11}(q_{11}-q_{22})(1-q_{22}^2q_{21}q_{12})x_{21}^3,\quad f_{10}=[x_{21},x_1].
    \end{eqnarray*}
    There does not exist $q$ such that $G(U_5,q)$ is a Gr\"{o}bner set. Indeed, consider $J(x_2,x_{221},x_{21})$, $J(x_2,x_{221},x_1)$ and $J(x_2,x_{22121},x_{21})$, one gets a system of equations with no solutions:
    \begin{eqnarray*}
    q_{21}q_{12}q_{11}=1,\quad q_{22}^2+q_{22}+1=0,\quad q_{11}=-q_{22}^3=-1,\quad q_{22}^6(q_{21}q_{12})^5q_{11}^4=1.
    \end{eqnarray*}

(5) $A(U'_5,q) \cong k\langle x_1,x_{21},x_{221},x_{2221},x_2\rangle/(f_1,f_2,\cdots,f_{10})$, where
    \begin{eqnarray*}
    &&f_1=[x_2,x_{2221}],\;\qquad f_2=[x_2,x_{221}]-x_{2221},\qquad f_3=[x_2,x_{21}]-x_{221},\\
    &&f_4=[x_2,x_1]-x_{21},\quad f_5=[x_{2221},x_{221}],\quad f_6=[x_{2221},x_{21}]-q_{22}^2q_{21}(q_{21}q_{12}q_{11}-1)x_{221}^2,\\
    &&f_7=[x_{2221},x_1]- \big(q_{22}q_{21}^2(q_{22}q_{21}q_{12}q_{11}+1)(q_{11}-q_{22}) + q_{21}^2q_{11}(1-q_{22}^4q_{21}q_{12})\big) x_{21}x_{221},\\
    &&f_8=[x_{221},x_{21}], \;\qquad f_9=[x_{221},x_1]- q_{21}(q_{11}-q_{22})x_{21}^2,\qquad f_{10}=[x_{21},x_1].
    \end{eqnarray*}
    $G(U_5',q)$ is a Gr\"{o}bner set iff $q_{22}=q_{11}=1$ and $q_{21}q_{12}=1$. In this case, $A(U_5',q)$ corresponds to $\mathcal{D}(q_{21})$ in \cite[Example 3.8]{ZL}. It is AS-regular of global dimension $5$ and of Gorenstein parameter $11$.

(6) $A(U''_5,q) \cong k\langle x_1,x_{211},x_{21},x_{221},x_2\rangle/(f_1,f_2,\cdots,f_{10})$, where
    \begin{eqnarray*}
    &&f_1=[x_2,x_{221}],\;\quad f_2=[x_2,x_{21}]-x_{221},\quad f_3=[x_2,x_{211}],\quad f_4=[x_2,x_1]-x_{21},\\
    &&f_5=[x_{221},x_{21}],\quad f_6=[x_{221},x_{211}],\;\;\qquad f_7=[x_{221},x_1]- q_{21}(q_{11}-q_{22})x_{21}^2,\\
    &&f_8=[x_{21},x_{211}],\quad f_9=[x_{21},x_1]-x_{211},\quad f_{10}=[x_{211},x_1].
    \end{eqnarray*}
    $G(U_5'',q)$ is a Gr\"{o}bner set iff one of the following conditions holds:
    \begin{enumerate}
    \item[(a)] $q_{22}=1,\, q_{11}=1,\, \, q_{21}q_{12}=1$.
    \item[(b)] $q_{22}=\zeta,\, q_{11}=\zeta,\, q_{21}q_{12}=\zeta^2$, where $\zeta^2+\zeta+1=0$.
    \end{enumerate}
    In Case (a) and Case (b), $A(U_5'',q)$ correspond to $\mathcal{L}(-2q_{21},\frac{1}{2},\frac{1}{2})$ and $\mathcal{L}(q_{21}\zeta^2,-1,q)$ where $q^2-q+1=0$,
    in \cite[Example 7.11]{ZL} respectively. They are AS-regular of global dimension $5$ and of Gorenstein parameter $10$.
\end{example}

\begin{proposition}\label{Proposition-Fibonacci}
Assume the notations given in Example \ref{Example-Fibonacci}. Then, for any $r\geq 5$, there is no matrix $q$ such that $G(U_r,q)$ is a Gr\"{o}bner set.
\end{proposition}

\begin{proof}
By Part (4) of Example \ref{Example-AS-regular}, it suffices to show that there is no matrix $q$ such that $G(U_r,q)$ is a Gr\"{o}bner set for $r\geq 6$. Assume the notations given in Lemma \ref{Lemma-calcultion}. Then, for $r\geq6$, one has
$$A(U_r,q)\cong k\langle x_u \, |\, u\in U_r\rangle/(H).$$
Clearly, $[x_2,x_{22121}],\; [x_{22121},x_{21}],\; [x_{221},x_{22121}]-x_{22122121},\; [x_2,x_{21}]-x_{221}\in H$. Thus, modulo $H$,
\begin{align*}
J(x_2,x_{22121},x_{21})&=q_{22}^3q_{21}^2 x_{22121}x_{221}-q_{22}^3q_{21}^3q_{12}^2q_{11}^2x_{221}x_{22121}\\
                       &= q_{22}^3q_{21}^2\big(1-q_{22}^6(q_{21}q_{12})^5 q_{11}^4\big)x_{22121}x_{221} - q_{22}
^6q_{21}^4q_{12}^3q_{11}^2x_{22122121}\\
                       &\neq 0.
\end{align*}
Therefore, $G(U_r,q)$ is not a Gr\"{o}bner set for any $r\geq 6$.
\end{proof}

\begin{remark}
The existence of AS-regular algebras of which the set of obstructions is $\Phi(U_r)$ is of particular interesting, for it relates to that whether or not the estimation in Corollary \ref{Corollary-bound-Gorenstein-parameter} is optimal. The existence of such algebras is realized in Parts (1), (2), (3) of Example \ref{Example-AS-regular} for $r=2,3,4$, and in \cite[Example 4.5]{ZL} for $r=5$. However, \cite[Proposition 8.1]{GF} tells that there is no such bigraded algebras for $r=6$. Naturally we have two questions:
\begin{enumerate}
\item Does there exist a bigraded AS-regular algebra with two generators of which the set of obstructions is $U_r$ for $r\geq 7$?
\item Does there exist an AS-regular algebra with two generators of which the set of obstructions is $U_r$ for $r\geq 6$?
\end{enumerate}
With Proposition \ref{Proposition-Fibonacci} at hand, we may get a negative answer for the first question.
\end{remark}

Now we turn to another approach which can be thought of as a generalization of the universal enveloping algebras of Lie algebras.

Let $c:k\langle X\rangle \otimes k\langle X\rangle  \to k\langle X\rangle \otimes k\langle X\rangle $ be the linear map given by $c(u\otimes v) = q_{u,v}v\otimes u.$
Then $k\langle X\rangle \otimes k\langle X\rangle$ is an associative algebra with multiplication $(\mu\otimes \mu) \circ (\id\otimes c\otimes \id)$ and unit $1\otimes 1$, where $\id=\id_{k\langle X\rangle}$ and $\mu$ is the multiplication of $k\langle X\rangle$. This algebra will be denoted $k\langle X\rangle \otimes^c k\langle X\rangle$. Set $\Delta: k\langle X\rangle \to k\langle X\rangle \otimes^c k\langle X\rangle$ and $\varepsilon:k\langle X\rangle \to k$ to be the homomorphisms of algebras given by $\Delta(x_i)=x_i\otimes 1 + 1\otimes x_i$ and $\varepsilon(x_i)=0$ for every $x_i\in X$. The structure $k\langle X\rangle^c = (k\langle X\rangle, \Delta, \varepsilon, c)$ is a braided bialgebra \cite[Definition 2.2]{Ar}. By definition, a polynomial $f\in k\langle X\rangle^c$ is called \emph{primitive} if $\Delta(f)=f\otimes 1+1\otimes f$.

\begin{theorem}\label{Theorem-approach-2}
Let $U$ be a finite closed set of Lyndon words and let $A=k\langle X\rangle/(G)$, where $G$ is a set of homogeneous polynomials such that $\Phi(U)\subseteq \lw(G) \subseteq \mathbb{L}\backslash U$. If $G$ is a Gr\"{o}bner set and it consists of primitive elements in $k\langle X\rangle^c$ such that
$
c(\,kG\otimes k\langle X\rangle + k\langle X\rangle \otimes kG\,) \subseteq kG\otimes k\langle X\rangle + k\langle X\rangle \otimes kG$.
Then $A$ is AS-regular of global dimension $d=\#(U)$ and of Gorenstein parameter $l=\sum_{u\in U} \deg(u)$. Moreover, $A$ is strongly Noetherian, Auslander-regular and Cohen-Macaulay.
\end{theorem}

\begin{proof}
By \cite[Theorem 4.3]{Ar}, $(G)$ is a braided biideal of $k\langle X\rangle^c$ and so $A$ is a braided bialgebra. Therefore, by \cite[Corollary 2]{Kh}, a super-letter $[u]$ is hard modulo $G$ iff $u$ is irreducible modulo $G$. Thus there are no hard super-letters other than $[u],u\in U$. Now Lemma \ref{Lemma-constructing-AS-regular} gives the result.
\end{proof}

\begin{remark}Assume the notations given in Theorem \ref{Theorem-approach-2}. We consider the special case when $q_{i,j}=1$ for all $1\leq i,j\leq n$. Then $c$ is the usual flip map, the bracket $[-,-]$ is the Lie operation and the set of primitive elements in $k\langle X\rangle^c$ is the free Lie algebra $\text{Lie}(X)$. Note that $A=k\langle X\rangle/(G)$ is the universal enveloping algebra of the positively graded Lie algebra $\mathfrak{g}=\text{Lie}(X)/(G)_L$. By \cite[Theorem 5.6, Theorem 5.8]{BC}, the set $\{\ [u]\;|\;u\in U\ \}$ is a basis of $\mathfrak{g}$. Thereby, in this case, Theorem \ref{Theorem-approach-2} is the well-known result that the universal enveloping algebra of a finite-dimensional positively graded Lie algebra is AS-regular.
\end{remark}

\begin{lemma}\label{Lemma-primitive}
Assume that $q_{i,j}q_{j,i}=1$ for any $1\leq i,j\leq n$. Then a polynomial $f\in k\langle X\rangle^c$ is primitive iff it is a linear combination of super-letters.
\end{lemma}

\begin{proof}
Note that $q_{u,v}q_{v,u}=1$ for any two words $u,v$. The result follows from an essentially the same discussion of that for \cite[Theorem 5.3.13]{Lo}.
\end{proof}

Combine Lemma \ref{Lemma-primitive} and Theorem \ref{Theorem-approach-2}, we conclude:

\begin{corollary}
Assume that $q_{i,j}q_{j,i}=1$ for any $1\leq i,j\leq n$. Let $U$ be a finite closed set of Lyndon words and $A=k\langle X\rangle/(G)$, where $G$ is a set of homogeneous polynomials such that $\Phi(U)\subseteq \lw(G) \subseteq \mathbb{L}\backslash U$. If $G$ is a Gr\"{o}bner set in which every polynomial is a linear combination of super-letters of the same constitute. Then $A$ is AS-regular of global dimension $d=\#(U)$ and of Gorenstein parameter $l=\sum_{u\in U}^d \deg(u)$. Moreover, $A$ is strongly Noetherian, Auslander-regular and Cohen-Macaulay.
\end{corollary}

\end{document}